\documentclass[12pt,bezier]{article}
\usepackage{amssymb}
\usepackage{mathrsfs}
\usepackage{amsmath}
\usepackage{amsfonts,amsthm,amssymb}
\usepackage{amsfonts}
\usepackage{graphics}
\newtheorem{lem}{Lemma}
\newtheorem{thm}[lem]{Theorem}
\newtheorem{cor}[lem]{Corollary}

\newtheorem{rem}[lem]{Remark}

\begin{document}
\title{A note on edge degree and spanning trail containing given edges}

\author{
Weihua Yang$^a$\footnote{Corresponding author. E-mail:
ywh222@163.com (W.~Yang).}, \quad Hongjian Lai$^b$,\quad Baoyindureng Wu$^c$\\
\small $^a$Department of Mathematics, Taiyuan University of
Technology, Taiyuan 030024, China\\
\small $^b$Department of Mathematics, West Virginia University,
Morgantown, WV 26506, United States\\
\small $^c$College of Mathematics and Systems Science, Xinjiang
University, Urumqi 830049, China}

\maketitle

\newcommand{\Aut}{\hbox{\rm Aut}}

\date{}

\maketitle {\flushleft\bf Abstract:} {\small  Let $G$ be a simple
graph with $n\geq4$ vertices and $d(x)+d(y)\geq n+k$ for each edge
$xy\in E(G)$.   In this work we prove that $G$ either contains a
spanning closed trail containing  any given edge set $X$ if $|X|\leq
k$, or $G$ is a well characterized graph. As a corollary, we show
that line graphs of such graphs are $k$-hamiltonian.}



\section{Introduction}
Unless stated otherwise, we follow \cite{bondy} for terminology and
notations, and
 consider finite connected simple  graphs.  For a
vertex $v\in V(G)$, $N_G(v)$ denotes  the set of the neighbors of
$v$ in $G$ and for  $A\subseteq V(G)$, $N_G(A)$  denotes  the set
$(\bigcup_{v\in A}N_G(v))\setminus A$.

Catlin in \cite{catlin} introduced the notion of a collapsible
graph. For a graph $G$, let $O(G)$ denote the set of odd degree
vertices of $G$. A graph $G$ is $eulerian$ if it has a closed trail
through all edges of  $G$. It is well-known (and easily seen) that
this is equivalent to $G$ being connected with $O(G)=\emptyset$, and
$G$ is $supereulerian$ if $G$ has a spanning eulerian subgraph. A
graph $G$ is $collapsible$ if for any subset $R \subseteq V (G)$
with $|R| \equiv 0 (\mod 2 )$, $G$ has a spanning connected subgraph
$H_R$ such that $O(H_R) = R$. Note that when $R =\emptyset$, a
spanning connected subgraph $H$ with $O(H)=\emptyset$ is a spanning
eulerian subgraph of $G$. Thus every collapsible graph is
supereulerian. Catlin \cite{catlin} showed that any graph $G$ has a
unique subgraph $H$ such that every component of $H$ is a maximally
collapsible connected subgraph of $G$ and every non-trivial
collapsible subgraph of $G$ is contained in a component of $H$. For
a subgraph $H$ of $G$, the graph $G/H$ is obtained from $G$ by
identifying the two ends of each edge in $H$ and then deleting the
resulting loops. The contraction $G/H$ is called the $reduction$ of
$G$ if $H$ is the maximal collapsible subgraph of $G$, i.e. there is
no non-trivial collapsible subgraphs in $G/H$. A vertex in $G/H$ is
called {\em non-trivial} if the vertex is obtained by contracted a
non-trivial connected collapsible subgraph. A graph $G$ is $reduced$
if it is the reduction of itself.

Catlin in \cite{catlinspanning} showed the following:

\begin{thm}[Catlin~\cite{catlinspanning}]\label{catlintrail}
Let $G$ be a connected simple graph on $n$ vertices and let $u,v\in
V(G)$. If
\begin{equation} \label{eq:a}
d(x)+d(y)\geq n
 \end{equation}
for each edge $xy\in E(G)$, then exactly one of the following holds:

$(i)$ $G$ has a spanning $(u,v)$-trail.

$(ii)$ $d(z)=1$ for some vertex $z\not\in\{u,v\}$.

$(iii)$ $G=K_{2,n-2},u=v$ and $n$ is odd.

$(iv)$ $G=K_{2,n-2},u\not=v, uv\in E(G)$, $n$ is even, and
$d(u)=d(v)=n-2.$

$(v)$ $u=v$, and $u$ is the only vertex with degree $1$ in $G$.
\end{thm}


The following stronger result strengthens Catlin's theorem above :

\begin{thm}[Li and Yang~\cite{liyang}]\label{mainresult1}
Let $G$ be a connected simple graph of order $n\geq4$ and $G$
satisfies $(\ref{eq:a})$ for every edge of $G$. Then exactly one of
the following holds:

 $(i)$  $G$ is collapsible.

 $(ii)$ The reduction of $G$ is $K_{1,t-1}$ for $t\geq3$ such that
all of the vertices of degree 1 are trivial and they have the same
neighbor in $G$, $t\leq\frac{n}2$. Moreover, if $t=2$, then
$G-\{v\}$ is collapsible for a vertex $v$ in the $K_2$.

 $(iii)$ $G$ is $K_{2,n-2}$.
\end{thm}

By the definition of the collapsible graph,
Theorem~\ref{mainresult1} implies Theorem~\ref{catlintrail}.

We call a  closed trail of $G$ is dominating closed trail if the
trail contains at least one vertex of each edge of $G$.
Theorem~\ref{mainresult1} also implies a previous
result~\cite{brualdi} by using the following theorem:

\begin{thm}[Harary and Nash-Williams~\cite{harary}]\label{nash}
Let $G$ be a graph with at least 4 vertices. The line graph $L(G)$
is hamiltonian if and only if $G$ ha a dominating closed trail.

\end{thm}

 Note that every graph in $\{K_{2,n-2},K_{1,n-1}\}$ has a closed trail that contains
at least one vertex of each edge of $G$. Then by
Theorem~\ref{mainresult1} we have:

\begin{thm}[Brualdi and Shanny~\cite{brualdi}]\label{brualdi}
Let $G$ be a connected graph with $n\geq4$ vertices.  If $G$
satisfies $(\ref{eq:a})$ for every edge of $G$, then $L(G)$ is
hamiltonian.
\end{thm}

Theorem~\ref{brualdi} was improved by Clark in~\cite{clark} as
follows:

\begin{thm}[Clark~\cite{clark}]\label{clark}
Let $G$ be a connected graph on $n\geq6$ vertices, and let $p(n)=0$
for $n$ even and $p(n)=1$ for $n$ odd. If each edge of $G$ satisfies
\begin{equation} \label{eq:b}
d(x)+d(y)\geq n-1-p(n),
 \end{equation}
 then $L(G)$ is hamiltonian.
\end{thm}

Motivated by Theorem~\ref{clark}, the authors of \cite{liyang}
considered replacing  the condition (\ref{eq:a}) in
Theorem~\ref{catlintrail} by (\ref{eq:b}), and classified the two
exceptional cases that can arise. The following result implies the
theorem above.

\begin{thm}[Li and Yang \cite{liyang}]\label{liyang}
Let $G$ be a connected graph of order $n\geq4$ and let $p(n)=0$ for
$n$ even and $p(n)=1$ for $n$ odd.  If each edge of $G$ satisfies
$(\ref{eq:b})$,
 then exactly one of the following holds:

$(i)$ $G$ is collapsible.

$(ii)$  The reduction of $G$ is $K_{1,t-1}$ for $t\geq3$ such that
all of the vertices of degree 1 are trivial and they have the same
neighbor in $G$, $t\leq\frac{n}2$. Moreover, if $t=2$, then
$G-\{v\}$ is collapsible for some vertex $v$ in the $K_2$.

 $(iii)$ $G$ is in $\{C_5,G_7,G_7',K_{1,n-1},K_{2,n-2},K_{2,n-3}'\}$, where
$K_{2,n-3}'$ is obtained from $K_{2,n-3}$ by adding  a pendant edge
on one of the vertices of degree $n-3$, and $G_7,G_7'$ are the
graphs shown in Figure.1.
\end{thm}

\begin{center}
\scalebox{0.7}{\includegraphics{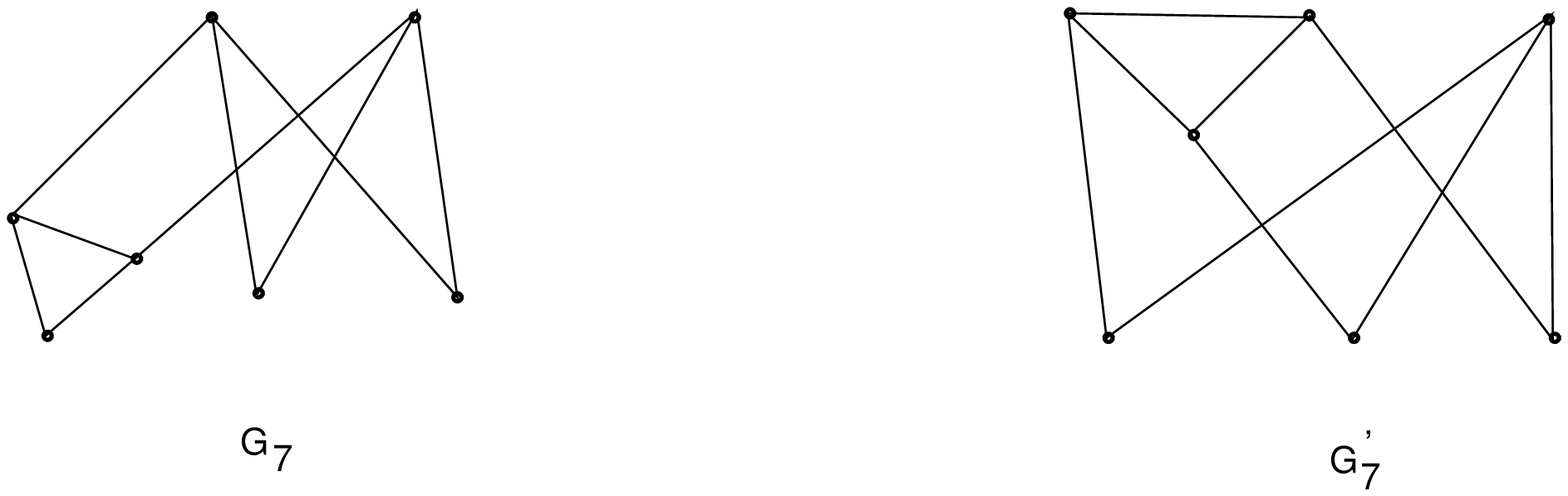}}

Figure 1. $G_7$ and $G_7'$.
\end{center}

in this note, we  strengthen  the edge-degree condition to consider
the spanning trail containing given edges. We obtain the following.

\begin{thm}\label{thm7}
Let $G$ be a connected simple graph on $n$ vertices, $X\subset E(G)$
satisfying $|X|\leq k$ and $k\geq 1$. If
\begin{equation*}
d(x)+d(y)\geq n+k
 \end{equation*}
for each edge $xy\in E(G)$, then $G$ either has a spanning closed
trail containing $X$, or $G$ is the graph obtain by adding an edge
on the two vertices of degree $n-2$ in $K_{2,n-2}$ for some even
number $n$ (denoted by $K_{2,n-2}^*$).
 \end{thm}

A graph $G$ is $k$-hamiltonian if $G-X$ is hamiltonian for any
$X\subset V(G)$ and $|X|\leq k$. As a corollary of the proof for the
above theorem, one can easily find the following.

\begin{cor}\label{hamilton}\label{cor8}
For a graph $G$, if \begin{equation*}  d(x)+d(y)\geq n+k
 \end{equation*}
for each edge $xy\in E(G)$, then $L(G)$ is $k$-hamiltonian.
\end{cor}

\section{Proof of Theorem~\ref{thm7}}

Let $G$ be a graph and let $X\subset E(G)$. The graph $G_X$ is
obtained from $G$ by replacing each edge $e\in X$ with ends $u_e$
and $v_e$ by a $(u_e, v_e)$-path $P_e$ of length 2, where the
internal vertex $w(e)$ of the path $P_e$ is newly added. By the
definition of supereulerian graphs, one can see that if  $G_X$ is
supereulerian, then $G$ contains a spanning even subgraph (closed
trail) through the edges in $X$.

We need several previous results due to Catlin concerning
collapsible graphs as follows:

\begin{thm}[Catlin~\cite{catlin}]\label{catlin} Let $G$ be a connected graph.  Each of the
following holds.

$(i)$ If $H$ is a collapsible subgraph of $G$, then $G$ is
collapsible if and only if $G/H$ is collapsible; $G$ is
supereulerian if and only if $G/H$ is supereulerian.

$(ii)$ A graph $G$ is reduced if and only if $G$ contains no
non-trivial collapsible subgraphs. As cycles of length less than 4
are collapsible, a reduced graph does not have a cycle of length
less than 4.

$(iii)$ If $G$ is reduced and if $|E(G)|\geq3$, then $\delta(G)
\leq3$, and $2|V (G)|-|E(G)| \geq 4$.

$(iv)$ If $G$ is reduced and violates $(iii)$, then $G$ is either
$K_1$ or $K_2$.
\end{thm}

Assume $H$ is a collapsible subgraph of $G$. It is easy to see that
if a vertex $y\in V(G-H)$ and $|N(y)\cap V(H)|\geq2$, then the
subgraph induced by $\{y\}\cup V(H)$ is  collapsible. We denote by
$\xi(G)$ the value $\min\{d(x)+d(y):xy\in E(G)\}$,  and call
$d(x)+d(y)$ the edge degree of $xy$, denoted by $d_{G}(xy)$. We
denote $G_X'$ the reduction of $G_X^*$, where $G_X^*$ is the graph
obtained from $G_X$ by contracting the edges with two ends of degree
2.

Clearly, if $G-X$  is collapsible, then $G_X^*$ is collapsible and
then $G$ contains a spanning even subgraph containing $X$. We next
consider the reduction of $G_X^*$.

 Since $\xi(G)\geq n+k$, we have $\delta(G)\geq k+1$.

\begin{lem}\label{preserve}
Let $G$ be a graph on $n$ vertices satisfying  $\xi(G)\geq n+k,k\geq 1$, and $X\subset E(G), |X|\leq k$.  Then $\xi(G-X)\geq n$.
\end{lem}

\begin{proof}
Sine $G$ is simple, $d_{G-X}(x)+d_{G-X}(y)\geq d_G(x)+d_G(y)-k\geq n$ for any edge $xy$. Thus, the claim holds.
\end{proof}

\begin{lem}\label{preserve}
Let $G$ be a graph on $n$ vertices satisfying  $\xi(G)\geq n$.
Suppose $G'$ is the reduction of $G$. Then $\xi(G')\geq |V(G')|$.
\end{lem}

\begin{proof}
Assume that $u$ is a non-trivial vertex of $G'$, i.e. it is the
contraction of a maximal collapsible connected subgraph $H$. We call
$H$ the $preimage$ of $u$ and denote $PM(u)=H$.
  It is sufficient to consider the edges that
are incident to $u$ in $G/PM(u)$ and the others are clear. Assume
that $yu\in E(G/PM(u))$. Then there is a vertex $x\in V(PM(u))$ such
that $xy\in E(G)$. Note that the neighbors of $x$ out of $PM(u)$ are
also neighbors of $u$ in $G'$. Then by $d(x)+d(y)\geq n$ we have
\begin{equation} \label{eq:1}
\begin{split}
d_{G/PM(u)}(u)+d_{G/PM(u)}(y)&\geq d(x)-(|PM(u)|)+1+d(y)\\
&\geq n-|PM(u)|+1\\
&=|G/PM(u)|
\end{split}
 \end{equation}

The assertion follows from inequality~(\ref{eq:1}).
\end{proof}

Combining the two Lemmas above, Theorem~\ref{mainresult1} holds for
$G-X$.

Clearly, if $G-X$ satisfies $(i), (ii)$ of
Theorem~\ref{mainresult1}, then one can easily see that $G_X^*$ is
collapsible. We assume $G-X$ is $K_{2,n-2}$. Note that $\xi(G)\geq
n+k$. Since $|X|\leq k$ and $G$ is simple, the edge-degree
$d_{G-X}(xy)$ satisfying $d_{G-X}(xy)\geq d_G(xy)-k$ for each edge
$xy$. So each edge in $|X|$ must incident to all the edges in $G-X$
and $|X|=k$. Note that $G-X$ is $K_{2,n-2}$. Then $|X|=1$ and $G$ is
the graph obtained by adding an edge on the two vertices of degree
$n-2$ in $K_{2,n-2}$. $\hfill\Box$

\section{Proof of Corollary~\ref{cor8}} By the arguments above, one can see that there is a
dominating closed trail in $G-X$ for any $X\subset E(G)$ and
$|X|\leq k$. Thus, Corollary~\ref{hamilton} holds.$\hfill\Box$

\begin{rem}
Theorem~\ref{thm7} cannot be generalized to graphs allowed
multi-edges. The works in \cite{catlinspanning} and \cite{liyang}
are both hold for simple graphs, but it does not hold for graphs
allowed multi-edges. The reason is easily to obtain, we omit the
discussion.
\end{rem}

\section{Acknowledgements}
The research is supported by NSFC (No.11671296, 61502330),  PIT of Shanxi, Research Project Supported by Shanxi Scholarship Council of China, and Fund Program for the Scientific Activities of Selected
Returned Overseas Professionals in Shanxi Province.

\end{document}